\documentclass[12 pt, reqno]{amsart}
\usepackage{mathrsfs}
\usepackage{graphicx}
\usepackage{tikz}
\usepackage{amsmath}
\usepackage{amsfonts}
\usepackage{amssymb}
\usepackage{hyperref}

\newtheorem{theorem}{Theorem}
\newtheorem{lemma}[theorem]{Lemma}

\newtheorem{proposition}[theorem]{Proposition}

\usepackage{amsthm,amsmath,amssymb,url,cite,color}

\numberwithin{equation}{section}

\def\lec{\operatorname{lec}}
\def\inv{\operatorname{inv}}
\def\cont{\operatorname{cont}}
\def\exc{\operatorname{exc}}
\def\des{\operatorname{des}}
\def\maj{\operatorname{maj}}
\def\Cont{\operatorname{Cont}}

\def\S{\mathfrak{S}}

\begin{document}

\title[A symmetrical  $q$-Eulerian identity]{A  symmetrical $q$-Eulerian identity}

\author{Guo-Niu Han}
\address[Guo-Niu Han]{I.R.M.A. UMR 7501, Universit\'e de Strasbourg et CNRS, France}
\email{guoniu.han@unistra.fr}

\author{Zhicong Lin}
\address[Zhicong Lin]{Department of Mathematics and Statistics, Lanzhou University, China,
and
 Institut Camille Jordan, UMR 5208 du CNRS, Universit\'{e} de Lyon, Universit\'{e} Lyon 1, France}
\email{lin@math.univ-lyon1.fr}

\author{Jiang Zeng}
\address[Jiang Zeng]{Institut Camille Jordan, UMR 5208 du CNRS, Universit\'{e} de Lyon, Universit\'{e} Lyon 1, France}
\email{zeng@math.univ-lyon1.fr}

\date{}

\begin{abstract}
We find a $q$-analog of the following symmetrical identity involving binomial coefficients $\binom{n}{m}$ and Eulerian numbers $A_{n,m}$,
due to Chung, Graham and Knuth [{\it J. Comb.}, {\bf 1} (2010),
29--38]:
\begin{equation*} 
\sum_{k\geq 0}\binom{a+b}{k}A_{k,a-1}=\sum_{k\geq 0}\binom{a+b}{k}A_{k,b-1}.
\end{equation*}
We 
give  two proofs, using generating function and   bijections, respectively.
\end{abstract}

\maketitle

%First page headline in LaTeX for S\'eminaire Lotharingien de Combinatoire
%--first part
\thispagestyle{myheadings}
\font\rms=cmr8
\font\its=cmti8
\font\bfs=cmbx8

\markright{\its S\'eminaire Lotharingien de
Combinatoire \bfs 67 \rms (2012), Article~B67c\hfill}
\def\thepage{}

\section{Introduction}
The \emph{Eulerian polynomials} $A_n(t)$ are defined by the exponential generating function 
\begin{align}\label{eq:def1}
\sum_{n\geq 0}A_n(t)\frac{z^n}{n!}=\frac{(1-t)e^{z}}{e^{zt}-te^{z}}.
\end{align}
The classical 
\emph{Eulerian numbers} $A_{n,k}$ are the coefficients  of the polynomial  $A_n(t)$, i.e.,  $A_n(t)=\sum_{k=0}^{n}A_{n,k} t^k$.
Recently, Chung, Graham and Knuth~\cite{cgk} noticed that if we modify the value of $A_0(t)$, which is 1 by \eqref{eq:def1},  by taking the convention that
$A_0(t)=A_{0,0}=0$, then 
the following symmetrical identity  holds:
\begin{equation} \label{eq: 0}
\sum_{k\geq 0}\binom{a+b}{k}A_{k,a-1}=\sum_{k\geq 0}\binom{a+b}{k}A_{k,b-1}\qquad (a,b>0).
\end{equation}
Equivalently,  instead of \eqref{eq:def1},  we  define the Eulerian polynomials by the generating function
\begin{align}\label{eq:def1bis}
\sum_{n\geq 0}A_n(t)\frac{z^n}{n!}=\frac{(1-t)e^{z}}{e^{zt}-te^{z}}-1=\frac{e^{z}-e^{tz}}{e^{zt}-te^{z}}.
\end{align}
At the end of \cite{cgk}, the authors   
asked for, among other unsolved problems, a $q$-analog of \eqref{eq: 0}.
The aim  of this paper is to give   such an extension  and provide  two proofs, of which one is analytical and  another  one is combinatorial.

We first introduce some $q$-notations. The $q$-shifted factorial $(z;q)_n$ is defined by\break
$(z;q)_n :=\prod_{i=0}^{n-1}(1-zq^i)$ for any positive integer $n$ and $(z;q)_0=1$. The
$q$-exponential function $e(z;q)$ is defined by
$$
e(z;q) :=\sum_{n\geq 0}\frac{z^n}{(q;q)_n}. 
$$
Several $q$-analogs of \eqref{eq:def1} have been proposed in the literature (see \cite{sw}).
Inspired by the recent work of 
Shareshian and Wachs \cite{sw},  we consider  the following $q$-analog of \eqref{eq:def1bis}:
\begin{equation}\label{eq: main}
\sum_{n\geq0}A_{n}(t,q)\frac{z^n}{(q;q)_n}=\frac{e(z; q)-e(tz; q)}{e(tz; q)-t\,e(z;q)}.
\end{equation}
The $q$-\emph{Eulerian polynomials} $A_{n}(t,q)$ have 
many remarkable properties analogous to Eulerian polynomials,
 see   Shareshian and Wachs \cite{sw}  and Foata and Han \cite{fh2}.
The $q$-\emph{Eulerian numbers} $ A_{n,k}(q)$ are then defined by
$$
A_{n}(t,q)=\sum_{k=0}^n  A_{n,k}(q) t^k \qquad (n\geq 0).
$$
The first  few terms  of $A_{n,k}(q)$ are 
$$
A_{0,0}(q)=0,\;
A_{1,0}(q)=1, \;A_{2,0}(q)=1, \; A_{2,1}(q)=1,\;A_{3,0}(q)=1,\;
A_{3,1}(q)=2+q+q^2,$$
and $A_{3,2}(q)=1$. Also, 
replacement of $t$ by $t^{-1}$ and of $z$ by $tz$ in \eqref{eq: main} yields that $t^nA_n(t^{-1}, q)=tA_n(t,q)$. Thus we have the 
  symmetrical  property
\begin{equation} \label{eq:sy}
A_{n,k}(q)=A_{n,n-k-1}(q).
\end{equation}
Recall that the $q$-binomial coefficients 
$\left[{\smallmatrix n\\ k\endsmallmatrix}\right]_q$ are defined by 
$$
{\bmatrix n\\ k\endbmatrix}_q:=\frac{(q;q)_n }{(q;q)_{n-k}(q;q)_k}\qquad \textrm{for}\quad 0\leq k\leq n,
$$
and $\left[{\smallmatrix n\\  k\endsmallmatrix}\right]_q=0$ if $k<0$ or $k>n$. 

The following symmetrical identity  involving both  the $q$-binomial coefficients $\left[{\smallmatrix n\\  k\endsmallmatrix}\right]_q$ and 
$q$-Eulerian numbers $A_{n,k}(q)$  is a true  $q$-analog of~\eqref{eq: 0}.
\begin{theorem} \label{th: 1}
For any positive integers $a$ and $b$,
we have the $q$-symmetrical identity
\begin{equation} \label{eq: th1}
\sum_{k\geq 0}{\bmatrix a+b\\  k\endbmatrix}_qA_{k,a-1}(q)=\sum_{k\geq 0}{\bmatrix a+b\\  k\endbmatrix}_qA_{k,b-1}(q).
\end{equation}
\end{theorem}

We shall first give a generating function proof of \eqref{eq: th1} in Section~2 and then a combinatorial proof in Section~3.  
We conclude the paper with some further extensions and remarks.

%First page headline in AmS-LaTeX for S\'eminaire Lotharingien de Combinatoire
%--restoring the headers and pagenumbering
\pagenumbering{arabic}
\addtocounter{page}{1}
\markboth{\SMALL GUO-NIU HAN, ZHICONG LIN, AND JIANG ZENG}{\SMALL A 
SYMMETRICAL $q$-EULERIAN IDENTITY}

\section{A generating function proof of \eqref{eq: th1}}
 It follows  from \eqref{eq: main} that 
 \begin{align}\label{eq:lem1}
\bigl(e(tz; q)-t\,e(z; q)\bigr)\sum_{n\geq0}A_{n}(t,q)\frac{z^n}{(q;q)_n}
%=e(z; q)-e(tz; q)
=\sum_{k}\frac{(1-t^k)z^k}{(q;q)_k}.
\end{align}
Now,
\begin{align*}
e(tz; q)\sum_{n\geq0}A_{n}(t,q)\frac{z^n}{(q;q)_n}&=\sum_{k}\frac{(tz)^k}{(q;q)_k}\sum_{n,i}A_{n,i}(q)t^i\frac{z^n}{(q;q)_n}\\
&=\sum_{k,n,i}{\bmatrix n+k\\ k\endbmatrix}_qA_{n,i}(q)t^{i+k}\frac{z^{n+k}}{(q;q)_{n+k}}\\
&=\sum_{k,n,i}{\bmatrix n\\  k\endbmatrix}_qA_{n-k,i-k}(q)t^{i}\frac{z^{n}}{(q;q)_n},
\end{align*}
and
\begin{align*}
t\,e(z;q)\sum_{n\geq0}A_{n}(t,q)\frac{z^n}{(q;q)_n}&=t\sum_{k}\frac{z^k}{(q;q)_k}\sum_{n,i}A_{n,i}(q)t^i\frac{z^n}{(q;q)_n}\\
&=\sum_{k,n,i}{\bmatrix n+k\\ k\endbmatrix}_qA_{n,i}(q)t^{i+1}\frac{z^{n+k}}{(q;q)_{n+k}}\\
&=\sum_{k,n,i}{\bmatrix n\\  k\endbmatrix}_qA_{n-k,i-1}(q)t^{i}\frac{z^{n}}{(q;q)_n}.
\end{align*}
Substituting the last two expressions in \eqref{eq:lem1} and identifying the coefficients of $t^iz^n/(q;q)_n$ on both sides,
we obtain 
$$
\sum_k{\bmatrix n\\  k\endbmatrix}_qA_{n-k,i-k}(q)-\sum_k {\bmatrix n\\ k\endbmatrix}_qA_{n-k,i-1}(q)=\begin{cases}
1,&\textrm{if $i=0\not =n$},\\
-1,&\textrm{if $i=n\not =0$},\\
0,&\textrm{otherwise}.
\end{cases}
$$
Setting $i=a$, $n=a+b$, and using the symmetrical property~\eqref{eq:sy}, we obtain~\eqref{eq: th1}. 
\qed
\medskip

We can also derive $q$-analogs of other identities in \cite{cgk}. For example, let
$H_n(t;q)=\sum_{i=0}^n\bmatrix n\\  i \endbmatrix_qt^i$ be the Rogers-Szeg\H o polynomials, see \cite[p.~49]{an}. Then
\begin{align*}
e(tz;q)e(z;q)
&=\sum_{n\geq 0}\frac{z^n}{(q;q)_n}H_n(t;q).
\end{align*}
If we multiply  \eqref{eq: main} by $e(tz;q)^2-t^2e(z;q)^2 =(e(tz;q)+te(z;q))(e(tz;q)-te(z;q))$,  then the right-hand side is equal to
\begin{align*}
(e(tz;q)+te(z;q))&(e(z;q)-e(tz;q))\\
&=\sum_{n\geq 0}\frac{z^n}{(q;q)_n} \left((1-t)H_n(t;q)+(t-t^n)H_n(1;q)\right).
\end{align*}
On the other hand, we have
\begin{align*}
(e(tz;q))^2\sum_{n\geq0}A_{n}(t,q)\frac{z^n}{(q;q)_n}
=\sum_{n, i,k}H_k(1;q) {\bmatrix n \\  k\endbmatrix}_q A_{n-k,i-k}(q)t^i\frac{z^n}{(q;q)_n},
\end{align*}
and 
\begin{align*}
t^2(e(z;q))^2\sum_{n\geq0}A_{n}(t,q)\frac{z^n}{(q;q)_n}
=\sum_{n, i,k}H_k(1;q){\bmatrix n \\  k\endbmatrix}_qA_{n-k,i-2}(q)t^i\frac{z^n}{(q;q)_n}.
\end{align*}
Hence,  identifying the coefficients of $t^iz^n/(q;q)_n$ in these expressions,
we arrive at
\begin{multline*}
\sum_kH_k(1;q){\bmatrix n \\  k\endbmatrix}_qA_{n-k,i-k}(q)
-\sum_kH_k(1;q){\bmatrix n \\  k\endbmatrix}_qA_{n-k,i-2}(q)\\
={\bmatrix n \\  i\endbmatrix}_q-{\bmatrix n \\  i-1\endbmatrix}_q+
\begin{cases}
H_n(1;q),&\textrm{if $i=1\not =n$},\\
-H_n(1;q),&\textrm{if $i=n\not =1$},\\
0,&\textrm{otherwise}.
\end{cases}
\end{multline*}

Of course, we can also  multiply  \eqref{eq: main} by $e(tz;q)^r-t^re(z;q)^r$ for any integer $r\geq 1$ and 
 derive a $q$-analog of the more general  (and complicated)  identity in \cite{cgk}. 

%%%%%%%%%%%%%%%%%%%%%%%%%
\section{A combinatorial proof  of \eqref{eq: th1}}%
%%%%%%%%%%%%%%%%%%%%%%%%%
%\subsection{Some  preliminaries}
For each permutation $\pi =\pi_1\pi_2\dots \pi_n$ of $[n]:=\{1,\ldots, n\}$, define the following four statistics:
\begin{align*}
&\exc(\pi) := |\{i: 1\leq i \leq n, \pi_i>i\}|;\\
&\des(\pi) := |\{i: 1\leq i \leq n-1, \pi_i>\pi_{i+1}\}|;\\
&\maj(\pi) := \sum_{\pi_i>\pi_{i+1}} i; \\
&\inv(\pi) := |\{(i, j): i<j, \pi_i>\pi_j\}|;
\end{align*} 
called number of \emph{excedances}, number of \emph{descents}, \emph{major index} and \emph{inversion number}, respectively. It is well-known that the Eulerian number  $A_{n,k}$ counts the number of permutation of $[n]$ with $k$ descents or $k$ excedances. 

Let $\S_n$ be the set of permutations of $[n]$. Shareshian and Wachs~\cite{sw} prove  that
$$
A_{n}(t,q)=\sum_{\pi\in \S_n}q^{(\maj-\exc)\pi}t^{\exc\pi}.
$$
For our purpose we shall use another interpretation of $A_{n}(t,q)$ due to 
Foata and Han~\cite{fh,fh2}.  This interpretation  is  based on Gessel's hook factorization of permutations \cite{ge}, which we recall now.
A word $w=x_1x_2\ldots x_m$ is called a \emph{hook} if $x_1>x_2$ and either $m=2$, or $m\geq 3$ and $x_2<x_3<\ldots<x_m$. Clearly,  each permutation $\pi =\pi_1\pi_2\dots \pi_n$ admits a unique factorization, called its \emph{hook factorization}, $p\tau_{1}\tau_{2}\dots \tau_{r}$, where $p$ is an increasing word and each factor $\tau_1$, $\tau_2$, \ldots, $\tau_k$ is a hook. 
To derive the hook factorization of a permutation, one can start from the right and factor out  each hook step by step. 
For each $i$ let $\inv\tau_i$ denote the number of inversions of $\tau_i$ and define
\begin{equation}
\lec(\pi) := \sum_{1\leq i\leq k}\inv (\tau_i).
\end{equation}
 For example, the hook factorization of $\pi=1\,3\,4\,14\,12\,2\,5\,11\,15\,8\,6\,7\,13\,9\,10$ is
 $$1\,3\,4\,14\,|12\,2\,5\,11\,15\,|8\,6\,7\,|13\,9\,10.$$
  Hence  $p=1\,3\,4\,14$, $\tau_1=12\,2\,5\,11\,15$, $\tau_2=8\,6\,7$, 
  $\tau_3=13\,9\,10$ and
   $$\lec(\pi)=\inv(12\,2\,5\,11\,15)+\inv(8\,6\,7)+\inv(13\,9\,10)=7.$$
 
 Let $p\tau_{1}\tau_{2}\dots \tau_{r}$ be the hook factorization of a permutation $\pi$. Let ${\mathcal A}_{0}$ (respectively ${\mathcal A}_{i}$, 
$1\leq i \leq r$) denote the set of all letters in the
word $p$ (respectively in the hook $\tau_{i}$). We call ${\mathcal A}_{0}=\cont(p)$ (respectively ${\mathcal A}_{i}=\cont(\tau_{i})$ ($1\leq i \leq r$)) the \emph{content} of $p$ (respectively of the hook $\tau_{i}$) and \emph{content} of $\pi$ the sequence $\Cont(\pi)=({\mathcal A}_{0}, {\mathcal A}_{1}, \dots , {\mathcal A}_{r})$. The statistic $(\inv- \lec)\pi$ is equal to the number of pairs $(k, l)$ such that $k\in {\mathcal A}_{i}$, $l\in {\mathcal A}_j$, $k> l$ and $i< j$, a number we shall denote by $\inv({\mathcal A}_{0}, {\mathcal A}_{1}, \dots , {\mathcal A}_{r})$. 

From Foata and Han\cite{fh,fh2}, we derive the following combinatorial interpretation:
\begin{align*}
A_{n}(t,q)
%=\sum_{\pi\in \S_n}q^{(\maj-\exc)\pi}t^{\exc\pi}
=\sum_{\pi\in \S_n}q^{(\inv-\lec)\pi}t^{\lec\pi}.
\end{align*}
Therefore
\begin{equation}\label{eq:key}
A_{n,k}(q)=\underset{\lec\pi=k}{\sum_{\pi\in\S_n} }q^{(\inv-\lec)\pi}.
\end{equation}
Recall \cite{an} that  the $q$-multinomial coefficient
\begin{equation*}
{\bmatrix n\\  a_0,a_1,\ldots,a_k\endbmatrix}_{q}=\frac{(q;q)_n }{(q;q)_{a_0}(q;q)_{a_1}\cdots(q;q)_{a_k}}
\end{equation*}
has the interpretation
\begin{align}\label{eq:qmul}
{\bmatrix n\\  a_0,a_1,\ldots,a_k\endbmatrix}_{q}=\sum_{({\mathcal A}_0, {\mathcal A}_1,\ldots, {\mathcal A}_k)}q^{\inv({\mathcal A}_0, {\mathcal A}_1,\ldots, {\mathcal A}_k)},
\end{align}
where the sum is over all ordered partitions $({\mathcal A}_0, {\mathcal A}_1,\ldots, {\mathcal A}_k)$ of $[n]$ such that $|{\mathcal A}_i|=a_i$, $0\leq i\leq k$.

We will  give a combinatorial proof of  \eqref{eq: th1} using \eqref{eq:key} and \eqref{eq:qmul}.
As a warm-up, we first prove the symmetry property \eqref{eq:sy} by constructing an explicit involution on permutations.

\begin{lemma}\label{th:2}
There is an involution  $\pi \mapsto \sigma$ on $\S_n$ satisfying
\begin{align*}
\lec (\pi) = n - 1 - \lec (\sigma),\quad \textrm{and}\quad
(\inv-\lec) {\pi} = (\inv-\lec) {\sigma}.
\end{align*}
\end{lemma}
\begin{proof}
Let $\tau$ be a hook  with $\inv(\tau)=k$ and  $\cont(\tau)=\{x_1, \ldots, x_m\}$, where  $x_1<\ldots <x_m$. 
Define
$d(\tau)=x_{m-k+1}x_1\ldots x_{m-k} x_{m-k+2}\ldots x_m$.
Clearly $d(\tau)$    is the unique hook 
 satisfying  $\cont(d(\tau))=\cont(\tau)$ and $\inv(d(\tau))=m-k=|\cont(\tau)|-\inv(\tau)$.

  Let $\tau$ be a word with $\inv(\tau)=k$ and  $\cont(\tau)=\{x_1, \ldots, x_m\}$, where  $x_1<\ldots <x_m$. 
Define
 $d'(\tau)=x_{m-k}x_1\ldots x_{m-k-1} x_{m-k+1}\ldots x_m$.
Clearly  $d'(\tau)$  is the unique word
 satisfying  $\cont(d'(\tau))=\cont(\tau)$ and $\inv(d'(\tau))=m-k-1=|\cont(\tau)|-\inv(\tau)-1$.
  
Let $\pi=p\tau_{1}\tau_{2}\ldots\tau_{r}$ be the hook factorization of $\pi\in \S_n$.
\begin{itemize}
\item If  $p\neq\emptyset$, let $\sigma=d'(p)d(\tau_{1})d(\tau_{2}),\ldots , d(\tau_{r})$. 
\item If $p=\emptyset$, let $\sigma=d'(\tau_1) d(\tau_{2})d(\tau_{3})\ldots d(\tau_{r})$. 
\end{itemize}
Since $d$ and $d'$ are two involutions, 
it is routine to check that such a mapping is an involution with the  required properties.
\end{proof}

%\subsection{Combinatorial proof of the symmetrical q-eulerian identity}
%Again, our combinatorial proof of \eqref{eq: th1} is based on the statistics
%$(\inv, \lec)$. 
For each fixed positive integer~$n$, a {\it two-pix-permutation of $[n]$} is
a  sequence of words
 \begin{equation}
\label{eq:defvec}
{\bf v}=(p_1, \tau_1, \tau_2, \ldots, \tau_{r-1}, \tau_{r}, p_2)
\end{equation}
satisfying  the following conditions:
\begin{itemize}
\item[(C1)] $p_1$ and $p_2$ are two increasing words, possibly empty;
\item[(C2)] $\tau_1, \ldots, \tau_r$ are hooks for some positive integer $r$;
\item[(C3)] The concatenation $p_1 \tau_1 \tau_2 \ldots \tau_{r-1} \tau_{r} p_2$ of all components of $\bf v$ is a permutation of $[n]$.
\end{itemize}

We also extend the two statistics to the two-pix-permutations by
\begin{align*}
\lec ({\bf v}) &= \inv(\tau_1)+ \inv(\tau_2) + \cdots + \inv(\tau_r), \\
\inv ({\bf v}) &=\inv(p_1 \tau_1 \tau_2 \ldots \tau_{r-1} \tau_{r} p_2).
\end{align*}
It follows  that
\begin{equation}\label{eq:il}
(\inv-\lec) {\bf v} = 
\inv(\cont(p_1), \cont(\tau_1),\cont(\tau_2),\ldots, \cont(\tau_r), \cont(p_2) ).
\end{equation}

\begin{lemma} \label{th:3}
The generating function of all two-pix-permutations ${\bf v}$ 
of $n$ such that $\lec({\bf v})=s$ by the statistic $\inv-\lec$ is
\begin{equation} \label{eq:comb}
\sum_{k\geq 0}{\bmatrix n\\ k\endbmatrix}_qA_{k,s}(q).
\end{equation}
\end{lemma}
\begin{proof}
By the hook factorization, the  two-pix-permutation 
$ {\bf v}$ in \eqref{eq:defvec}
%=(p_1, \tau_1, \tau_2, \ldots, \tau_{r-1}, \tau_{r}, p_2)$
is in bijection with the pair 
$(\sigma, p_2)$, 
where 
$\sigma= p_1 \tau_1 \tau_2 \ldots \tau_{r-1} \tau_{r}$ is a permutation on $[n]\setminus\cont(p_2)$ and $p_2$ is an increasing word.  Thus, by \eqref{eq:key}, \eqref{eq:qmul}, and \eqref{eq:il}, 
the generating function of all two-pix-permutations ${\bf v}$ 
of $[n]$ such that $\lec({\bf v})=s$ and $|p_2|=n-k$ with respect to the weight
$q^{(\inv-\lec)({\bf v})}$ is
$\left[{\smallmatrix n\\ k\endsmallmatrix}\right]_qA_{k,s}(q)$.
\end{proof}

\begin{lemma}\label{th: 4}
There is a bijection ${\bf v} \mapsto {\bf u}$ 
on the set of all two-pix-permutations of $[n]$ satisfying
\begin{align*}
\lec ({\bf v}) = n - 2 - \lec ({\bf u}),\quad \textrm{and}\quad
(\inv-\lec) {\bf v} = (\inv-\lec) {\bf u}.
\end{align*}
\end{lemma}

\begin{proof}
We give an explicit construction of  the bijection. 
Let ${\bf v}$ be a two-pix-permutation and write
%$${\bf v}=(p_1, \tau_1, \tau_2, \ldots, \tau_{r-1}, \tau_{r}, p_2)$$
%be a two-pix-permutation. 
%We write
$${\bf v}=(\tau_0, \tau_1, \tau_2, \ldots, \tau_{r-1}, \tau_{r}, \tau_{r+1}),$$
where  $\tau_0=p_1$ and $\tau_{r+1}=p_2$.
If $\tau_i$ (respectively $\tau_j$) is  the leftmost (respectively rightmost) non-empty
word (clearly $i=0, 1$ and $j=r, r+1$), we can write ${\bf v}$ in the following compact way by removing the empty words at the beginning or at the end:
\begin{equation}
\label{eq:v_compacted}
{\bf v}=(\tau_i, \tau_{i+1},  \ldots, \tau_{j-1}, \tau_{j}).
\end{equation}
It is easy to see that the above procedure is reversible 
by adding some necessary empty words at the two  ends of the compact form \eqref{eq:v_compacted}. Now we work with the compact form. 
Recall that
\begin{equation}
\label{eq:inv-lec:compact}
(\inv-\lec) {\bf v} = 
\inv(\cont(\tau_i),\cont(\tau_{i+1}),\ldots, \cont(\tau_{j-1}), \cont(\tau_j) )
\end{equation}
and
$\lec ({\bf v}) = 
\sum_{k=i}^j \lec(\tau_k)
$.

If $i=j$, then only one word $\tau_i$ 
is in the sequence $\bf v$.
We define ${\bf u} = (\emptyset, \sigma_i, \emptyset)$, where
$\sigma_i$ is  the unique word (hook) with content $[n]$ such that 
$\lec(\sigma_i) = n-2 - \lec(\tau_i)$. 

If $j>i$,  we define the two-pix-permutation ${\bf u}$ by
$$
{\bf u}=(d'(\tau_i), d(\tau_{i+1}), d(\tau_{i+2}),  \ldots, d(\tau_{j-1}), d'(\tau_{j})),
$$
where $d$ and $d'$ are two involutions
defined in the proof of Lemma~\ref{th:2}.

Since $\lec (d'(\tau_i)) = |\cont(\tau_i)| - 1 - \lec(\tau_i)$,
$\lec (d'(\tau_j)) = |\cont(\tau_j)| - 1 - \lec(\tau_j)$ and
$\lec (d'(\tau_k)) = |\cont(\tau_k)|  - \lec(\tau_k)$  for $k\not=i,j$,
	we have
$$\lec ({\bf u}) = \sum  |\cont(\tau_k)| -2 - \lec ({\bf v})
=n -2 - \lec ({\bf v}).
$$
Finally it follows from \eqref{eq:inv-lec:compact}  that $(\inv-\lec){\bf u}=(\inv-\lec){\bf v}$.

We give an example to illustrate the bijection. Let ${\bf v} =(27, 6389, 514, \emptyset)$. Then 
 ${\bf v} $ is a two-pix-permutation of $[9]$ and
$\inv({\bf v})=19, \lec({\bf v})=3, (\inv-\lec){\bf v}=16$.
The compact form is $(27, 6389, 514)$, so that 
$$
{\bf u}= (d'(27), d(6389), d'(514) = (72, 9368 , 145) .
$$
Since the first word $72$ is  not increasing, we obtain the standard form by adding the empty word at the beginning, so that
${\bf u}= (\emptyset, 72, 9368 , 145)$.
Hence
$\inv({\bf u})=20$, $\lec({\bf u})=4$, and  $(\inv-\lec){\bf u}=16$.
\end{proof}

Combining Lemmas~\ref{th:2}, \ref{th:3} and \ref{th: 4}, we obtain a combinatorial proof of \eqref{eq: th1}.
%%%%%%%%%%%%%%%%%%%%%%
\section{Further extensions and remarks}
%%%%%%%%%%%%%%%%%%%%%

%\begin{prop}\label{th: 3}
The classical Eulerian polynomials correspond to the generating function of 
 descent numbers of symmetric groups. Let $r\geq 1$ be an integer. 
 As a natural extension of \eqref{eq: main}, we consider the polynomial $A_{n}^{(r)}(t,q)$  defined by the generating function
\begin{equation} \label{eq: 1b}
\frac{e(z;q^r)-e(t^rz;q^r)}{e(t^rz;q^r)-te(z;q^r)}=\sum_{n\geq 1}A_{n}^{(r)}(t,q)\frac{z^n}{(q^r;q^r)_n}.
\end{equation}
It is easy to see that $A_{n}^{(r)}(1,1)=r^n n!$, which is the cardinality of 
 $C_r\wr \mathfrak{S}_n$, that is, the 
wreath product  of a cyclic group $C_r$ of order $r$ with the symmetric group $\S_n$.
 Introduce the coefficients $A_{n,i}^{(r)}$ by
\begin{align}\label{eq:defr}
A_{n}^{(r)}(t,q)=\sum_{i}A_{n,i}^{(r)}(q)t^i.
\end{align}
The generating function proof of \eqref{eq: th1} can be applied to derive immediately the identity
\begin{equation} \label{eq: ma}
\sum_k{\bmatrix n\\ k\endbmatrix}_{q^r}A_{k,rn-i-1}^{(r)}(q)
=\sum_k{\bmatrix n\\ k\endbmatrix}_{q^r}A_{k,i-1}^{(r)}(q)
\end{equation}
for positive integers $i$ and $n$ such that $i\not= rn$. 
%\end{prop}

Let $r$ and $n$ be two positive integers. Let ${\mathcal A}=\{a,b,\ldots\}$ be any subset of $[n]$. We define the \emph{$r$-colored version} 
of ${\mathcal A}$ by 
$${\mathcal A}^r :=\{a^1,b^1,\ldots,a^2,b^2,\ldots,\ldots,a^r,b^r,\ldots\}.$$ 
Order the elements of $[n]^r$ by
\begin{equation*}
1^1<1^2\ldots<1^r<2^1<2^2\ldots<2^r<\ldots<n^1<n^2<\ldots<n^r.
\end{equation*}

A \emph{pix-$r$-colored-word} is a sequence of colored words
$$
{\bf w}=(p, \tau_1, \tau_2, \ldots, \tau_{k})
$$
satisfying the following conditions:
\begin{itemize}
\item[(C1)] $p$ is an increasing word, with content ${\mathcal A}_0^r$, the $r$-colored version of ${\mathcal A}_0$, possibly empty;
\item[(C2)] $\tau_i$ $(1\leq i\leq k)$ are hooks, with content ${\mathcal A}_i^r$, 
	the $r$-colored version of ${\mathcal A}_i$, and the positive integer $k$ is not fixed;
\item[(C3)] $({\mathcal A}_0,{\mathcal A}_1,\ldots,{\mathcal A}_k)$ is an ordered partition of $[n]$.
\end{itemize}
Let  $\mathscr{W}_{n,r}$  be  the set of all pix-$r$-colored-words of $[n]$. For each ${\bf w}\in \mathscr{W}_{n,r}$ we define two statistics by
\begin{align*}
\inv_r({\bf w}) &=\inv(p \tau_1 \tau_2 \ldots \tau_{r-1} \tau_{r}),\\
\lec_r(\bf w)&= \inv(\tau_1)+\cdots +\inv(\tau_k).
\end{align*}

\begin{proposition}
Let $A_{n}^{(r)}(t,q)$ be defined by \eqref{eq: 1b}.
Then
$$
\sum_{\sigma\in\mathscr{W}_{n,r}}q^{(\inv_r-\lec_r)(\sigma)}t^{\lec_r(\sigma)}=A_{n}^{(r)}(t,q).
$$

\end{proposition}
\begin{proof}
By definition and \eqref{eq:qmul}, we have
\begin{align*}
\sum_{\sigma\in\mathscr{W}_{n,r}}q^{(\inv_r-\lec_r)(\sigma)}t^{\lec_r(\sigma)}
&=\underset{ a_i\geq1}{\sum_{a_0+a_1+\ldots+a_k=n}}\prod_{1\leq i\leq k}P_{ra_i}(t)\underset{\#{\mathcal A}_i=a_i}{\sum_{({\mathcal A}_0,{\mathcal A}_1,\ldots,{\mathcal A}_k)}}q^{\inv_r({\mathcal A}_0,{\mathcal A}_1,\ldots,A_k)}\\
&=\underset{ a_i\geq1}{\sum_{a_0+a_1+\ldots+a_k=n}}\bmatrix n\\ a_0,a_1,\ldots,a_k\endbmatrix_{q^r}\prod_{1\leq i\leq k}P_{ra_i}(t),
\end{align*}
where $P_m(t) :=t+t^2+\ldots+t^{m-1}$.
So the generating function is 
\begin{align*} \label{eq: gf}
\sum_{n,i}&\sum_{\sigma\in\mathscr{W}_{n,r}}q^{(\inv_r-\lec_r)(\sigma)}t^{\lec_r(\sigma)}\frac{z^n}{(q^r;q^r)_n}\\
&=\sum_{n\geq 0}\underset{a_i\geq1}{\sum_{a_0+a_1+\ldots+a_k=n}}\bmatrix n\\ a_0,a_1,\ldots,a_k\endbmatrix_{q^r}\prod_{1\leq i\leq k}P_{ra_i}(t)\frac{z^n}{(q^r;q^r)_n}\nonumber \\ 
&=\left(\sum_{a\geq 0}\frac{z^{a}}{(q^r;q^r)_{a}}\right)\left(1-\sum_{b\geq 1}P_{rb}(t)\frac{z^b}{(q^r;q^r)_b}\right)^{-1}\nonumber\\
&=e(z;q^r)\left(1-\sum_{b\geq 1}\frac{t-t^{rb}}{1-s}\frac{z^b}{(q^r;q^r)_b}\right)^{-1}\nonumber\\
&=\frac{(1-t)e(z;q^r)}{e(t^rz;q^r)-t\,e(z;q^r)}.
\end{align*}
This completes the proof in view of \eqref{eq: 1b}.
\end{proof}

Similarly, we can define \emph{two-pix-$r$-colored-words} to be sequences of colored words 
$$
{\bf w}=(p_1, \tau_1, \tau_2, \ldots, \tau_{k},p_2)
$$
with the following conditions:
\begin{itemize}
\item[(C1)] $p_1$ and $p_2$ are increasing words, with content ${\mathcal A}_0^r$ and ${\mathcal B}_0^r$, possibly empty;
\item[(C2)] $\tau_i$ $(1\leq i\leq k)$ are hooks, with content ${\mathcal A}_i^r$, the $r$-colored version of ${\mathcal A}_i$, and the positive integer $k$ is not fixed;
\item[(C3)] $({\mathcal A}_0,{\mathcal A}_1,\ldots,{\mathcal A}_k,{\mathcal B}_0)$ is an ordered partition of $[n]$.
\end{itemize}
Clearly we can  give a combinatorial proof 
 of \eqref{eq: ma} by applying the following generalization of  Lemma~\ref{th: 4}.
\begin{proposition}\label{th: 5}
There is a bijection ${\bf v} \mapsto {\bf u}$ 
on the set of all two-pix-$r$-colored-words such that
\begin{equation*}
\lec_r {\bf v} = rn - 2 - \lec_r {\bf u},\quad \text{and}\quad
(\inv_r-\lec_r) {\bf v} = (\inv_r-\lec_r) {\bf u}.
\end{equation*}
\end{proposition}

  When $r=1$, through Gessel's hook factorization, 
  we can translate the statistic $(\inv_r,\break \lec_r)$ to  $\S_n$.  
  It would be interesting to see whether there is an analogous hook factorization 
  for general $r\geq 1$ so that we can translate our 
 $(\inv_r, \lec_r)$ defined on $\mathscr{W}_{n,r}$ to $C_r\wr \mathfrak{S}_n$.

We conclude this paper with another symmetric identity for the Eulerian numbers.
Notice that
for any positive integers $n$ and $k$ we have 
$$
{\bmatrix 2n\\ 2k+1\endbmatrix}_{-1}=0\quad\text{and}\quad  {\bmatrix 2n\\ 2k\endbmatrix}_{-1}={\binom nk}.
$$
It is  known  \cite[Cor.~6.2]{ssw} that if $dk=n$ and $\omega_d$ is a  primitive $d^{th}$ root of unity, then 
\begin{equation*}
A_{n}(t,\omega_d)=A_k(t)\left(\frac{1-t^d}{1-t}\right)^k.
\end{equation*} 
In particular,  if  $d=2$, then $\omega_d=-1$.  Hence,
assuming that $a+b$ is even, 
 the substitution $q=-1$ in  \eqref{eq: th1}   yields
\begin{equation}
\sum_{k\geq 0}{\binom {\frac {a+b}{2} }{ k}}\sum_{i+j=a-1}{\binom k  i}A_{k,j}=\sum_{k\geq 0}{\binom {\frac {a+b}{2} } k}\sum_{i+j=b-1}{\binom k  i}A_{k,j}.
\end{equation} 

  This identity can be rephrased in the form
\begin{equation*}
\sum_{k\geq 0}{\binom {c+d } k}\sum_{i+j=2c-1}{\binom k  i}A_{k,j}=\sum_{k\geq 0}{\binom {c+d } k}\sum_{i+j=2d-1}{\binom k  i}A_{k,j}
\end{equation*} 
and
\begin{equation*}
\sum_{k\geq 0}{\binom {c+d-1 } k}\sum_{i+j=2(c-1)}{\binom k  i}A_{k,j}=\sum_{k\geq 0}{\binom {c+d -1} k}\sum_{i+j=2(d-1)}{\binom k i}A_{k,j}
\end{equation*} 
for any positive integers $c$ and $d$.

The last  two symmetrical  identities involving binomial coefficients and Eulerian  numbers cry out for a combinatorial interpretation.

\subsection*{Acknowledgement}
This work was  partially supported by 
 the grant ANR-08-BLAN-0243-03.


\begin{thebibliography}{99}

\bibitem{an} George E. Andrews, The Theory of Partitions, Reading MA: Addison-Wesley, 1976.

\bibitem{cgk} Fan Chung, Ron Graham, Don Knuth, A symmetrical Eulerian identity, J.  Comb., \textbf{1} (2010), 29--38.

\bibitem{fh} Dominique Foata and Guo-Niu Han, The $q$-tangent and $q$-secant numbers via basic Eulerian polynomials, Proc. Amer. Math. Soc., \textbf{138} (2009), 385--393.

\bibitem{fh2} Dominique Foata and Guo-Niu Han, Fix-mahonian calculus III; a quadruple distribution, Monatsh. Math., \textbf{154} (2008), 177--197.


\bibitem{ge} Ira Gessel, A coloring problem, Amer. Math. Monthly, \textbf{98} (1991), 530--533.


\bibitem{sw} John Shareshian and Michelle L. Wachs, Eulerian quasisymmetric function, Adv. in Math., \textbf{46} (2011), 536--562.

\bibitem{sw2} John Shareshian and Michelle L. Wachs, $q$-Eulerian polynomials: excedance number and major index, Electron. Res. Announc. Amer. Math. Soc., \textbf{13} (2007), 33--45.
%
\bibitem{ssw} Bruce Sagan,  John Shareshian and Michelle L. Wachs, Eulerian quasisymmetric functions and cyclic sieving, Adv. in Appl. Math., \textbf{46} (2011), 536--562.

\end{thebibliography}
\end{document}